\documentclass[a4paper,14pt]{article} 
\usepackage[T2A]{fontenc}
\usepackage[utf8]{inputenc}
\usepackage[russian,english]{babel} 
\usepackage{amssymb,amsfonts,amsmath,mathtext,enumerate,float,amsthm} 
\usepackage[unicode,colorlinks=true,citecolor=black,linkcolor=black]{hyperref}
\usepackage{indentfirst} 
\usepackage[dvips]{graphicx} 
\graphicspath{{img/}}

\usepackage{geometry} 
\usepackage{cite}

\theoremstyle{plain}
\newtheorem{theorem}{Theorem}[section]

\newtheorem{lemma}[theorem]{Lemma}
\newtheorem{corollary}[theorem]{Corollary}

\theoremstyle{definition}
\newtheorem{definition}[theorem]{Definition}

\let\etoolboxforlistloop\forlistloop 
\usepackage{mathtools}
\mathtoolsset{showonlyrefs}
\let\forlistloop\etoolboxforlistloop 

\DeclareMathOperator{\diam}{diam}
\DeclareMathOperator{\chr}{char}

\begin{document}


\title{
	On the characteristic and diameter of planar integral point sets
	\footnote{
		The work of the first author was partially supported by the Theoretical Physics
		and Mathematics Advancement Foundation “BASIS”.
	}
}

\author{
	N.N. Avdeev
	\footnote{nickkolok@mail.ru, avdeev@math.vsu.ru},
	E.A. Lushina
	\\
	\textit{Voronezh State University}
}

\maketitle

\paragraph{Abstract.}
A point set $M$ in Euclidean plane is called an integral point set in semi-general position if all the distances between the
elements of $M$ are integers, and $M$ does not contain collinear triples.
We improve the lower bound for diameter of such sets in the particular case
when the characteristic of the set is of the form $4k+1$ or $4k+2$.
To achieve that, we combine hyperbolae-based and grid-based toolsets.

\section{Introduction}
	A planar integral point set (IPS) is a set of points in the plane, such that the distance between any pair of its points is an integer, and at least one triple of its points is non-collinear. The latter condition is essential to avoid subsets of a straight line; those are de-facto equivalent to subsets of integer numbers and form a completely different combinatorial object.

	In 1945, Erdös gave an elegant proof~\cite{anning1945integral,erdos1945integral} that every IPS is finite.
	In any IPS he chose a non-collinear triple $\{M_1, M_2, M_3\} \in M$ so that any other point $M_0 \in M$ lies on either the straight line $M_1M_2$, the perpendicular bisector to the segment $M_1M_2$, or one of $|M_1M_2|-1$ hyperbolae, where $|M_1M_2|$ stands for length of line segment $M_1M_2$.  Applying the same argument to the line segment $M_1M_3$, Erdös concluded that $\#M \le 4 \cdot |M_1M_2| \cdot |M_1M_3|$, where $\#M$ stands for cardinality of $M$.

	Thus, we can easily infer the lower bound for diameter of an IPS $M$:
	$$
		\operatorname{diam} M \ge \frac{\sqrt{\#M}}{2}.
	$$
	That was the first lower bound for the diameter ever;
	however, Erdös did not only estimate the diameter,
	but also established the toolset for further investigation, the core of which is a system of cofocal hyperbolae.

	Using this toolset, in 2003 Solymosi proved~\cite{solymosi2003note} that
	$$
		\operatorname{diam} M \ge c \cdot \#M.
	$$
	Although Solymosi did not give the constant $c$ explicitly,
	it can be inferred~\cite{our-vmmsh-2018-translit} from his proof that $c=\frac{1}{24}$.

	In~\cite{our-vmmsh-2018-translit} the constant (for $n\geq 4$) was improved to $0.3457$
	employing Point Packing in a Square Problem~\cite{markot2005newverified,costa2013valid}.
	In~\cite{my-pps-linear-bound-2019} the approach has been further developed, and the constant has been tightened to $\frac{5}{11}$.
	Finally, in~\cite{my-semi-general-5-4-bound-2019} for IPS $M$ in semi-general position (that is an IPS with no collinear triples) it was proved that
	\begin{equation}
		\label{eq:chebsb_lower_bound}
		\diam M \geq \left(\frac n 5 \right)^{5/4}
	\end{equation}
	(in assumption that the set has at east 4 points).
	All these lower bounds are based on the Erdos's framework --- cofocal hyperbolae.

	Meanwhile in 1988 Kemnitz introduced~\cite{kemnitz1988punktmengen}
	a characteristic of an IPS, that is a squarefree integer $q$ such as the area
	of a triangle formed by any triple of points from the IPS is comparable with $\sqrt{q}$.
	(Indeed, Kemnitz proved this fact for any point set with rational distances ---
	in contrast with IPS, those can be infinite even if they contain non-collinear triples,
	see~\cite{huff1948diophantine} for an example construction
	and~\cite{solymosi2010question} for some known limitations.)
	In 2000s, Kurz introduced~\cite{kurz2005characteristic} the function $d(2,n)$ that evaluates
	to the minimal possible diameter of planar IPS of cardinality $n$.
	Then Kurz employed Kemnitz's results and found~\cite[Subsection 4.2]{kurz2008minimum} exact values of $d(2,n)$
	up to $n=122$ by exhaustive computer search.
	(Taking the characteristic into consideration allows to boost such search significantly ---
	basically because all triples of an IPS form triangles with equal characteristic.)

	For generalization in higher dimensions, we refer the reader to~\cite{nozaki2013lower}.

	In the present paper, we put the power of Erdös's and Kemnitz's approaches together. Due to squarefree nature of characteristic, it can be of the form $4k+1$, $4k+2$ or $4k+3$, where $k$ is integer.
	We prove that for a particular case of characteristic $4k+1$ and $4k+2$
	the bound~\eqref{eq:chebsb_lower_bound} can be improved to
	\begin{equation}
		\operatorname{diam} M \ge \left( \frac{25}{36} n \right)^{5/4}
		.
	\end{equation}

	In Section 2, we give all the required notions and known results. In Section 3, we discuss some examples of integral point sets in order to demonstrate that none of the classes $4k+1$, $4k+2$, $4k+3$ is too exotic nor pathological. Section 4 is devoted to the connection between Erdös curves and characteristic. In Section 5, we prove some auxiliary results, and we proceed to the main one in Section 6.

\section{Basic Notions and Results}

In this Section, we provide rigorous definitions and list some known results.

For the sake of brevity, the following notations will be used for sets
of positive integers, non-negative integers and all the integers resp.:
\begin{equation}
	\mathbb{N} = \{1,2,3,4,...\},\quad \mathbb{N}_0 = \mathbb{N} \cup \{0\},
	\quad
	\mathbb{N}_\pm = \{0,\pm 1,\pm 2,\pm 3,\pm 4,...\}
\end{equation}

For a finite set $M$, we will denote its cardinality by $\#M$.

\begin{definition}
	\label{def:IPS}
	A planar integral point set (IPS) is a set $M$
	of non-collinear points in the plane $\mathbb{R}^{2}$ such that
	for any pair of points $M_{1}, M_{2} \in M$
	the Euclidean distance $|M_{1}M_{2}|$
	between points $M_{1}$ and $M_{2}$ is integral.
	Notation: $M\in\mathfrak{M}$, and also $M\in\mathfrak{M}_n$ for $n=\#M$.
\end{definition}

When we say that a set in non-collinear,
we mean that it has at least one non-collinear triple.
If we tighten the condition and require all the triples to be non-collinear,
we get the next definition.

\begin{definition}
	\label{def:IPS_semi_general}
	A planar IPS $M$ is said to be in \emph{semi-general position}
	if no three points of $M$ are collinear.
	Notation: $M\in\dot{\mathfrak{M}}$, and also $M\in\dot{\mathfrak{M}}_n$ for $n=\#M$.
\end{definition}

In the present paper, we mostly focus on IPS in semi-general position.
However, the next restriction step can be done as the following.

\begin{definition}
	\label{def:IPS_general}
	A planar IPS $M$ is said to be in \emph{general position}
	if no three points of $M$ are collinear
	and no four points of $M$ are concircular.
	Notation: $M\in\overline{\mathfrak{M}}$, and also $M\in\overline{\mathfrak{M}}_n$ for $n=\#M$.
\end{definition}

\begin{definition}
	The diameter of an integral point set $M$ is defined by setting
	\begin{equation}
		\operatorname{diam} M = \underset{M_{1}, M_{2} \in
		M}{\max} |M_{1}M_{2}|
		.
	\end{equation}
\end{definition}

\begin{definition}
	Two numbers $a$ and $b$ are \emph{commensurable} if their ratio $a/b$
	is a rational number.
\end{definition}

For example, the pair $(7, 2/3)$ is commensurable,
the pair $(\sqrt3/3, \sqrt{12})$ is also commensurable,
but the pair $(\sqrt2, \sqrt3)$ is not.

\begin{definition}
	A number is called \emph{squarefree} if its only perfect square divisor is 1.
	The first squarefree numbers are: 1, 2, 3, 5, 6, 7, 10, 11, 13, ...
\end{definition}

\begin{definition}
	\label{def:char}
	The characteristic of a planar IPS $\mathcal{M}$ is a squarefree number $q$
	such that the area of any triangle $M_1M_2M_3$, $\{M_1,M_2,M_3\}\subset M$,
	is commensurable with $\sqrt{q}$.
	Notation: $\chr M = q$.
\end{definition}

\begin{definition}
	If points $M_1, M_2 \in M \in \mathfrak{M}$,
	then the line segment $M_1 M_2$ is said to be an \emph{edge} of $M$.
\end{definition}

The following result is due to Kemnitz~\cite{kemnitz1988punktmengen}:

\begin{theorem}[the Grid Theorem]
	A set $M \in \mathfrak{M}_n$ with characteristic $p$ can be placed on the grid
	$$
	\left\{\left(\frac{a_i}{2m} ; \frac{b_i \sqrt{p}}{2m}\right)\right\},
	$$
	where $a_i, b_i \in \mathbb{N}_\pm$, and $m$ can be taken as the length of any edge of the set $M$.
	\label{tm:1}
\end{theorem}

For the sake of completeness, we sketch the proof for the Grid Theorem below.

\begin{proof}
	Let $M_1, M_2 \in M \in \mathfrak M$ and $|M_1 M_2| = m$.
	Set $M_1=(-m/2, 0)$, $M_2=(m/2, 0)$.
	Then for any $M_i \in M$, $M_i = (x,y)$ one has
	$|M_i M_1| = k \in \mathbb N_0$, $|M_i M_2| = n \in \mathbb N_0$.
	The point $M_i$ belongs to the intersection of two circles,
	whose equations are
	\begin{gather}
		\left(-\frac{m}{2} - x\right)^2 + y ^2 = k^2,
	\\
		\left( \frac{m}{2} - x\right)^2 + y ^2 = n^2,
	\end{gather}
	where $k+n\geq m$.

	The solution is
	\begin{gather}
		x = \frac{k^2 - n^2}{2 m} = \frac{a_i}{2m},
	\\
		y = \pm\sqrt{k^2 - \left(\frac{m}{2}+x\right)^2} =
		\pm\sqrt{k^2 - \left(\frac{m}{2}+\frac{a_i}{2m}\right)^2} =
		\frac{\pm b_i \sqrt{q}}{2m},
	\end{gather}
	and the claim follows.
\end{proof}

\begin{definition}
	Let $M_1 M_2$ be an edge of $M\in \mathfrak{M}$.
	For $N\in\mathbb N_\pm$, $|n| < |M_1 M_2|$
	we will say that the set of points
	\begin{equation}
		\{M_0 : |M_0 M_1| - |M_0 M_2| = n\}
	\end{equation}
	is called the \emph{$n$-th Erdos curve}.
\end{definition}

Obviously, the 0-th Erdos curve is the perpendicular bisector of $M_1 M_2$,
and all the other Erdos curves are branches of cofocal hyperbolae.
On Figure~\ref{fig:Erdos_curves_numbered}, the Erdos curves are shown for an edge of length 3.

\begin{figure}
\center{\includegraphics[width=0.95\linewidth]{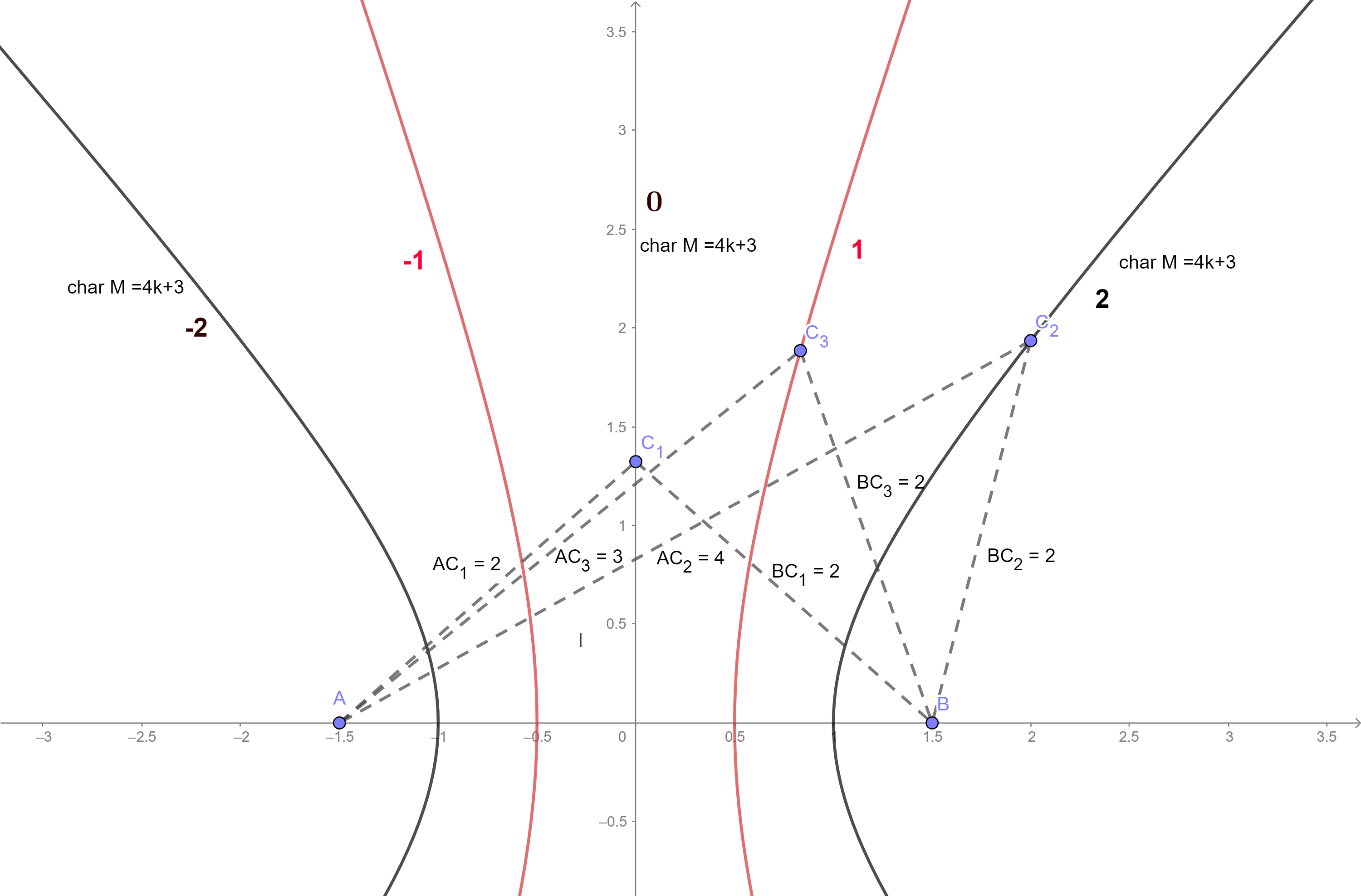}}
\caption{Erdos curves}
\label{fig:Erdos_curves_numbered}
\end{figure}

Thus, an edge $M_1 M_2$ generates $2|M_1 M_2| - 1$ Erdos curves.
For the sake of brevity, Erdos curves with odd numbers are named odd Erdos curves,
and the ones with even numbers are named even Erdos curves.

Two following definitions are used to classify integral point sets
with many collinear triples~\cite{avdeev2019particular}:

\begin{definition}
	A planar integral point sets of $n$ points with $n-1$ points on a straight line is called
	a \emph{facher} set.
\end{definition}
For $9\leq n\leq 122$, the minimal possible diameter is attained at a facher set~\cite{kurz2008bounds}.

\begin{definition}
	A planar integral point set situated on two parallel straight lines
	is called a \emph{rails} set.
\end{definition}

\begin{definition}
	The part of a plane between two parallel straight lines with distance $\rho$ between them
	is called a strip of width $\rho$.
\end{definition}

\begin{lemma}
	\cite{smurov1998stripcoverings}
	\label{lem:smurov_min_height}
	If a triangle $T$ with minimal height $\rho$ is situated in a strip,
	then the width of the strip is at least $\rho$.
\end{lemma}

\begin{lemma}
	\cite[Lemma 4]{our-vmmsh-2018-translit};
	\cite[Lemma 2.4]{my-pps-linear-bound-2019}
	\label{lem:square_container}
	Let $M\in\mathfrak{M}(2,n)$, $\operatorname{diam} M = d$.
	Then $M$ is situated in a square of side length $d$.
\end{lemma}

\begin{definition}
	\cite[Definition 2.5]{my-pps-linear-bound-2019}
	A \textit{cross} for points $M_1$ and $M_2$, denoted by $cr(M_1,M_2)$, is the union of two straight lines:
	the line through $M_1$ and $M_2$,
	and the perpendicular bisector of line segment $M_1 M_2$.
\end{definition}

\begin{lemma}
	\cite[Theorem 3.10]{my-pps-linear-bound-2019}
	\label{lem:no_distance_one}
	Each set $M\in\mathfrak{M}_n$
	such that for some $M_1,M_2 \in M$ equality $|M_1 M_2|=1$ holds,
	consists of $n-1$ points, including $M_1$ and $M_2$, on a straight line,
	and one point out of the line, on the perpendicular bisector of line segment $M_1 M_2$.
\end{lemma}

Our attention will be mostly restricted to planar integral point sets
with characteristic of the form $4k+1$ and $4k+2$ in semi-general position.
Thus, if an IPS $M$ satisfies these conditions,
we will write that
$M\in\overline{\mathfrak{M}}'$, and also $M\in\overline{\mathfrak{M}}'_n$ for $n=\#M$.

\section{Integral Point Sets with Various Characteristics}

Let us demonstrate that the classes $4k+1$, $4k+2$ nor $4k+3$
are neither too exotic nor pathological.
In order to do this, we will provide important examples of IPS for each class.

For convenience, we use the notation \cite{our-ped-2018,our-pmm-2018,our-vmmsh-2018}:
$\sqrt{p}/q * \{ (x_1,y_1), ...,$ $ (x_n, y_n)  \}$,
which means that each abscissa is multiplied by $1/q$
and each ordinate is multiplied by $\sqrt{p}/q$,  i.e.
\begin{equation}
	\label{eq:char_lattice}
	\sqrt{p}/q * \{ (x_1,y_1), ..., (x_n, y_n)  \}
	=
	\left\{ \left(\frac{x_1}{q},\frac{y_1\sqrt{p}}{q}\right), ..., \left(\frac{x_n}{q},   \frac{y_n\sqrt{p}}{q}\right)  \right\}
	,
\end{equation}
where $q$ is the characteristic of the IPS.
Such notation is made possible by the Grid Theorem.

When we think about IPS with characteristic $4k+1$,
the very first example that comes to our mind is Egyptian triangle,
that is the triangle with sides $(3;4;5)$.
Egyptian triangle is obviously an IPS with characteristic $1 = 4 \cdot 0 +1$.

Facher sets with characteristic 1, that are called \emph{semi-crabs},
are investigated in~\cite{antonov2008maximal}.
There are also much more complex IPS with characteristic 1;
for those from $\dot{\mathfrak{M}}_7$, we refer the reader to~\cite{kurz2013constructing}.

Figure~\ref{fig:char385} shows a rails IPS of characteristic $385 = 4 \cdot 96 + 1$ presented in~\cite{avdeev2019particular}.
Its coordinates are
\begin{multline}
	\sqrt{385}/{2} * \{ (\pm 1105, 48),
	(\pm 2189 ; 0),
	(\pm 1587 ; 0),
	(\pm 1269 ; 0),
	\\
	(\pm 763 ; 0),
	(\pm 623 ; 0),
	(\pm 529 ; 0),
	(\pm 339 ; 0)\}
\end{multline}

\begin{figure}
\center{\includegraphics[width=0.95\linewidth]{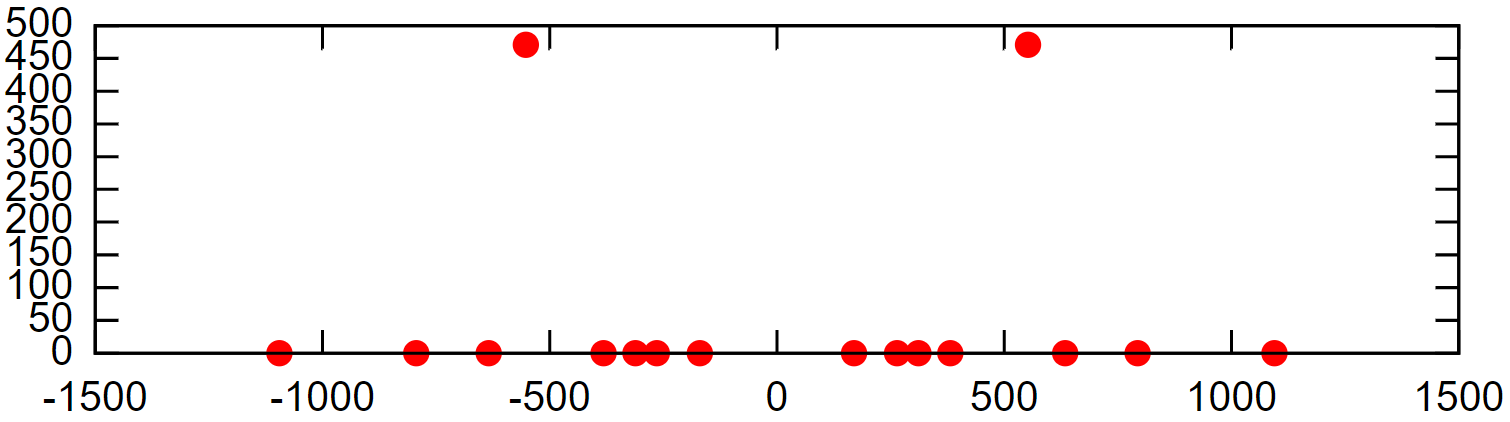}}
\caption{IPS of cardinality 16 and diameter 2189}
\label{fig:char385}
\end{figure}

In~\cite{kreisel2008heptagon}, the first ever known
planar integral point set $M_7 \in \dot{\mathfrak{M}}_7$ is given,
see Figure~\ref{fig:heptagon2}.
Its coordinates are
\begin{multline}
\sqrt{2002}/2227 * \{ (0;0), (2227^2\cdot10;0), (26127018;932064), (32 142 553; 411 864),
\\
	 (17615968;238464), (7 344 908;411 864), (19079044; 54168)\}
	 ,
\end{multline}
and
$
	\chr M_7 = 2002 = 4\cdot 500 + 2
$.

It is noticeable that the second example of an IPS from $\dot{\mathfrak{M}}_7$
given in the same article has the same characteristic.

Also the largest known rails IPS presented in~\cite{momot2022example}
with 104 points on one straight line and the rest 2 on another
(that gives the cardinality of 106) has the characteristic of $154 = 4 \cdot 38 + 2$.
(We do not list the coordinates of that set here due to its diameter which is 2745754098774581800288844387372160.)

\begin{figure}
\center{\includegraphics[width=0.95\linewidth]{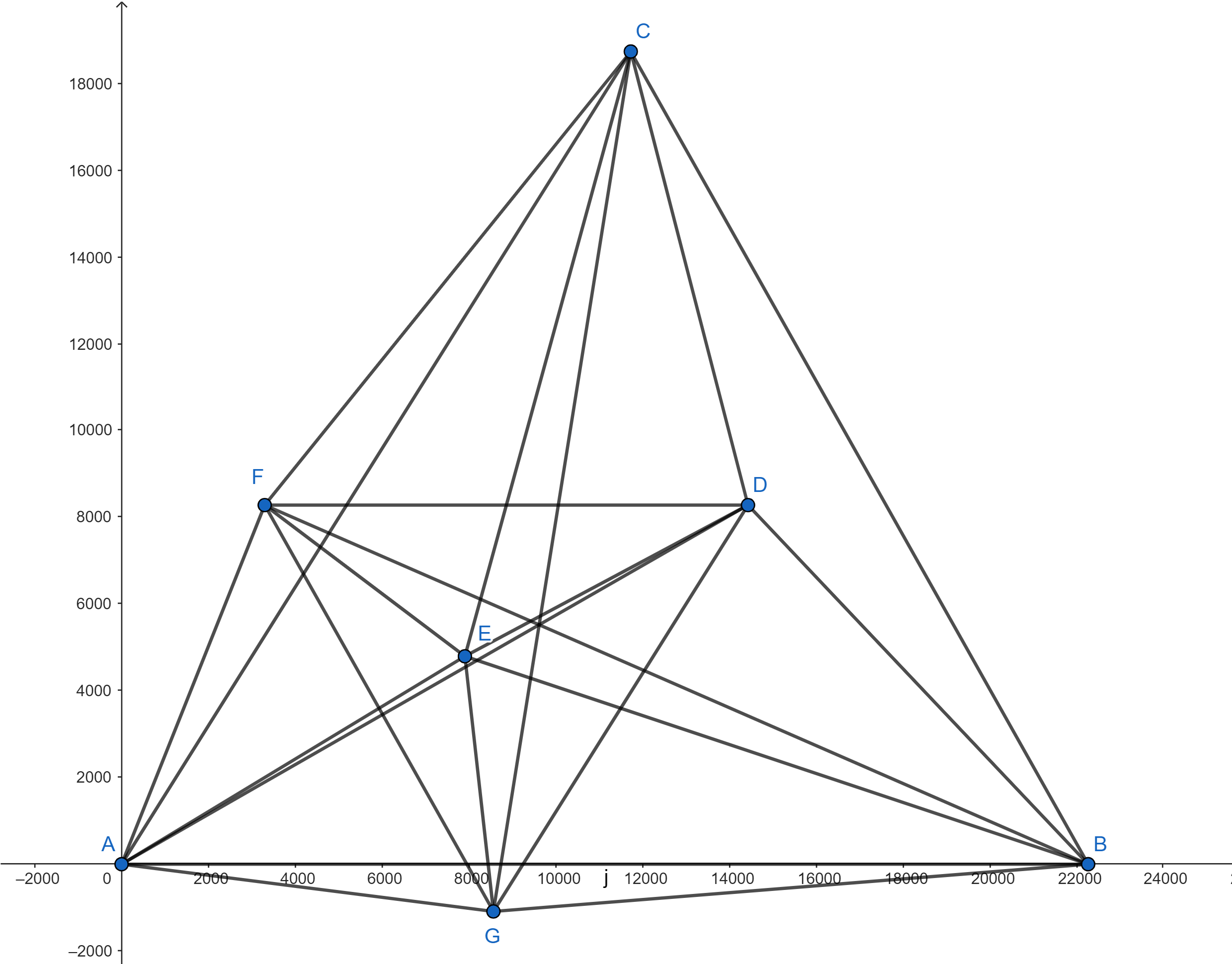}}
\caption{The heptagon in general position}
\label{fig:heptagon2}
\end{figure}

As for characteristic $4k+3$, we should mention that the \emph{upper} bound for the minimal diameter
of planar integral point set given in~\cite{harborth1993upper} employs IPS with characteristic $3$.
Moreover, all the IPS of minimal possible diameter provided in~\cite[\S 5, Figure 1]{harborth1993upper}
(for cardinalities from 3 to 9) have characteristic of form $4k+3$.

For the sake of completeness, on Figure~\ref{fig:11_2432_255255_a0a1b7f805cfc93080b4608505a0f45d_rails_3.png}
we give an example of rails set
with 3 points on one line and 8 points on the other (first presented in~\cite[Figure 1]{avdeev2019particular})
whose characteristic is
\begin{equation}
	255255 = 3 \cdot 5 \cdot 7 \cdot 11 \cdot 13 \cdot 17
	 = 4 \cdot 63813  + 3
\end{equation}
and whose coordinates are
\begin{multline}
\mathcal{P}_{3,8}=\sqrt{255255}/2*\{
	( 1767 ; -3);
	( 2791 ; -3);
	( 4071 ; -3);
\\
	( -306 ; 0);
	( 0 ; 0);
	( 1798 ; 0);
	( 2304 ; 0);
	( 2760 ; 0);
	( 3534 ; 0);
	( 4040 ; 0);
	( 4558 ; 0)
\}
\end{multline}

\begin{figure}
\center{\includegraphics[width=0.95\linewidth]{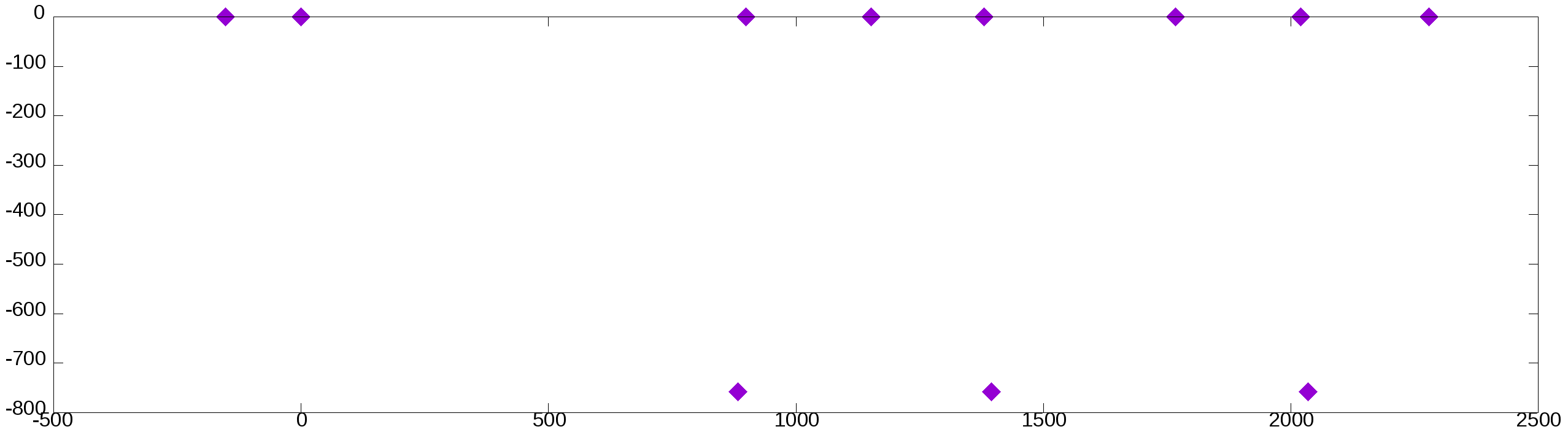}}
\caption{The rails set with characteristic 255255}
\label{fig:11_2432_255255_a0a1b7f805cfc93080b4608505a0f45d_rails_3.png}
\end{figure}

\section{Erdös Curves and Characteristic}

\begin{lemma}
	\label{lem:pre_weeding_I}
	Any set $M = \{A, B, C\} \in \mathfrak{M}_3$ with edge $|AB| = m$, where $|AC| - |BC| = m - s$, $m, s \in \mathbb{N}$, $s$ is an odd number and $s \leq m$, has a characteristic of the form $p = 4k + 3$, $k \in \mathbb{N}_0$.
\end{lemma}

\begin{proof}
	Let $|BC| = n$. Then, from the triangle inequality $|AC| < |AB| + |BC|$, we have that $|AC| = m + n - s$, $s \in \mathbb{N}$ and $s \leq m$.
	By the Grid Theorem, we can assume that $A = (-\frac{m}{2}; 0)$, $B = (\frac{m}{2}; 0)$, and $C = (\frac{a}{2m}; \frac{b\sqrt{p}}{2m})$. We find the distance between points $A$ and $C$, and points $B$ and $C$ in coordinates, and form a system of equations
	\begin{equation}
		\label{eq:29}
		\begin{cases}
		\frac{\sqrt{(a+m^2)^2+pb^2}}{2m}&=m+n-s.\\
		\frac{\sqrt{(a-m^2)^2+pb^2}}{2m}&=n
		\end{cases}
	\end{equation}
	Multiply each equation by $2m$ and square both sides of the equations in system~\eqref{eq:29}
	\begin{equation}
		\label{eq:30_pre_weeding_I}
		\begin{cases}
		(a+m^2)^2+pb^2&=4m^2(m+n-s)^2,\\
		(a-m^2)^2+pb^2&=4m^2n^2
		\end{cases}
	\end{equation}
	and then subtract the second equation from the first one:
	\begin{equation}
	(a+m^2)^2-(a-m^2)^2=4m^2(m+n-s)^2-4m^2n^2,
	\label{eq: 31}
	\end{equation}
	$$
	4m^2a=4m^2((m+n-s)^2-n^2),
	$$
	so
	\begin{equation}
		\label{eq:32_pre_weeding_I}
		a=(m-s)(m+2n-s).
	\end{equation}
	Substitute the obtained expression~\eqref{eq:32_pre_weeding_I}
	into the second equation of the system~\eqref{eq:30_pre_weeding_I}:
	\begin{equation}
		((m-s)(m+2n-s)-m^2)^2+pb^2=4m^2n^2,
	\end{equation}
	that immediately gives
	\begin{equation}
		\label{eq:33}
		pb^2=s(2m+2n-s)(2n-s)(2m-s).
	\end{equation}
	Since $s$ is assumed to be odd, let $s=2t+1$, where $t\in\mathbb{N}_0$.
	Then equation \eqref{eq:33} can be rewritten as:
	\begin{equation}
		\label{eq:34.2}
		pb^2=(2t+1)(2m+2n-2t-1)(2n-2t-1)(2m-2t-1)
	\end{equation}
	or
	\begin{eqnarray}
		pb^2+1 &=&4(-4t^4+8t^3m+8t^3n-8t^3-4t^2m^2-12t^2mn+12t^2m- {}\nonumber\\
			 &     &-4t^2n^2+12t^2n-6t^2+4tm^2n-4tmn^2-12tmn+6tm- {}\nonumber\\
			 &     &-4tn^2+6tn-2t+2m^2n-m^2+2mn^2-3mn+m-n^2+n).
	\label{eq:35}
	\end{eqnarray}
	If equation~\eqref{eq:35} has integer solutions, then both $p$ and $b$ are odd.
	The right-hand side is divisible by 4, which implies that  $pb^2\equiv3\operatorname{mod}{4}$.
	Since $b^{2}\equiv1(\operatorname{mod}4)$, it follows that $p=4k+3$, $k\in \mathbb{N}_{0}$.
\end{proof}

The following theorem follows immediately.

\begin{theorem}[The Weeding Theorem I]
	\label{thm:weeding_I}
	Let $M = \{M_1, M_2, M_3\} \in \mathfrak{M}_3$ be an IPS of semi-general position with a characteristic different from $4k+3$, $k \in \mathbb{N}_0$. Then, the length of the edge $M_1 M_2$ and its $n$-th Erdős curve, which contains the point $M_3$, have the same parity.
\end{theorem}

\begin{proof}
	Let $|M_1 M_2| = m$, $m \in \mathbb{N}$, and let the $n$-th Erdős curve satisfy $|M_1 M_3| - |M_2 M_3| = n$.
	If $m$ is an odd number, then by Lemma~\ref{lem:pre_weeding_I}, the difference $n = |M_1 M_3| - |M_2 M_3|$ equals $m - s$, where $s$ is an even number, is also an odd number.
	Let $m$ be an even number; then $n = |M_1 M_3| - |M_2 M_3| = m - s$ is an even number, as a difference of two even numbers.
\end{proof}

The case when the length of an edge is even is slightly more specific and allows a more precise statement.

\begin{lemma}
	\label{lem:pre_weeding_II}
	Every set $M = \{A, B, C\} \in \mathfrak{M}_3$ with an even edge $|AB| = 2q$, where $|AC| - |BC| = 2q - s$, $q, s \in \mathbb{N}$, $s$ is an odd number and $s \leq 2q$, has a characteristic of the form $p = 8k + 7$, $k \in \mathbb{N}_0$.
\end{lemma}

\begin{proof}
	Let us set $m = 2q$ in the Lemma~\ref{lem:pre_weeding_I}.
	Then equation \eqref{eq:33} turns into
	\begin{equation}
		\label{eq:33.1}
		pb^2 = s(4q + 2n - s)(2n - s)(4q - s).
	\end{equation}
	Taking into account that $s = 2t + 1$, $t \in \mathbb{N}_0$,
	we bring equation \eqref{eq:33.1} to the form:
	\begin{equation}
	pb^2 = (2t + 1)(4q + 2n - 2t - 1)(2n - 2t - 1)(4q - 2t - 1)
	\label{eq: 34.2}
	\end{equation}
	or, after expansion,
	\begin{eqnarray}
		pb^2 + 1&=& 8(-2t^4+8t^3q+4t^3n-4t^3-8t^2q^2-12t^2qn+12t^2q- {}\nonumber\\
			 &     &-2t^2n^2+6t^2n-3t^2+8tq^2n-4tqn^2-12tqn+6tq- {}\nonumber\\
			 &     &-2tn^2+3tn-t+4q^2n-2q^2+2qn^2-3qn+q)-  {}\nonumber\\
			   &     &-4n(n-1).
		\label{eq:35.1}
	\end{eqnarray}
	Among two consecutive natural numbers,
	exactly one is even,
	so $4n(n-1)\equiv0\pmod{8}$.
	Thus, the right-hand side of equation~\eqref{eq:35.1} is a multiple of 8.
	Then $pb^2\equiv7\pmod{8}$.
	A necessary condition for the existence of an integer solution to equation~\eqref{eq:35.1}
	is that numbers $p$ and $b$ must be odd.
	Since $b^2\equiv1\pmod{8}$, it follows that $p=8k+7$, $k\in\mathbb{N}_0$.
\end{proof}

\begin{theorem}[The Weeding Theorem II]
	\label{thm:weeding_II}
	Let an IPS in semi-general position $M=\{M_1, M_2, M_3\}\in{\mathfrak{M}_{3}}$ have an even edge length $|M_1 M_2|$, and let point $M_3$ lie on an odd Erdős curve of edge $M_1 M_2$.
	Then the set $M$ has characteristic $p=8k+7$, $k\in \mathbb{N}_{0}$.
\end{theorem}

\begin{proof}
	Indeed, let $|M_1 M_2|=2q$, $q \in \mathbb{N}$.
	For the $n$-th Erdős curve we have $|M_1 M_3|-|M_2 M_3|=n$.
	According to Lemma~\ref{lem:pre_weeding_II}, the difference $n=|M_1 M_3|-|M_2 M_3|=2q-s$,
	where $s$ is an odd number,
	is also an odd number and $\operatorname{char} M = 8k + 7$, $k\in \mathbb{N}_{0}$.
\end{proof}

\section{Auxiliary Results}

\begin{lemma}
	\label{lem:no_distance_2}
	Every set $M \in \overline{\mathfrak{M}_n}'$ with an edge $|M_1M_2|=2$ has a cardinality $n=3$.
\end{lemma}

\begin{proof}
	Let $M \in \overline{\mathfrak{M}_n}'$ and $M_1, M_2 \in M$.
	Then all points in the set $M$ lie on $cr(M_1, M_2)$.
	Otherwise, according to Lemma~\ref{lem:pre_weeding_II}, $\operatorname{char}{M}={8k+7}$, $k\in \mathbb{N}_0$.

	Suppose, to the contrary, that $M={M_1, M_2, M_3, M_4} \in \overline{\mathfrak{M}_4}'$.
	Employ the Grid Theorem and set $M_1(-1;0), M_2(1;0), O(0;0)$.
	Then $M_3$ and $M_4$ lie on the perpendicular bisector to $M_1M_2$.
	Since the distance $|M_3M_4|$ is an integer, the area of triangle $M_1M_3M_4$ is a rational number.
	By Definition~\ref{def:char}, this means that $\operatorname{char}{M}={1}$.
	Then by the Grid Theorem we have $M_3(0;t/4)$. Without loss of generality, let us assume that $t\in\mathbb{N}$. Let $|M_1M_3|=s$. Applying the Pythagorean theorem to triangle $OM_1M_3$, we get
	$$
	1+\frac{t^2}{16}=s^2,
	$$
	or equivalently,
	\begin{equation}
	\label{eq: 42}
	16s^2-t^2=16.
	\end{equation}
	Now, we find the solutions to equation~\eqref{eq: 42} in positive integers.
	To do this, we introduce the substitution: $t=4k$, $k\in\mathbb{N}$.
	Then equation~\eqref{eq: 42} takes the form:
	\begin{equation}
	\label{eq: 43}
	s^2-k^2=1.
	\end{equation}
	Break down the left-hand side of equation~\eqref{eq: 43} into factors:
	\begin{equation}
	\label{eq: 44}
	(s-k)(s+k)=1.
	\end{equation}
	Since $s,k\in\mathbb{N}$, the number 1 can be represented as the product of two integers in two ways: $1\cdot1$ and $-1\cdot(-1)$. In the first case, $(s-k)=1$ and $(s+k)=1$, which gives $s=1,k=0$, which contradicts that $k\in\mathbb{N}$.
	In the second case, $(s-k)=-1$ and $(s+k)=-1$, which gives $s=-1$, $k=0$, which also contradicts that $s,k\in\mathbb{N}$.

	Thus, equation~\eqref{eq: 42} has no positive integer solutions.
	The obtained contradiction completes the proof.
\end{proof}

The following result is a key element for bounds presented
in~\cite{solymosi2003note} and~\cite{my-semi-general-5-4-bound-2019}:

\begin{lemma}
	\cite[Observation 1]{solymosi2003note}
	If a triangle $T$ has integer side lengths $a \leq b \leq c$,
	then its minimal height $m$ is at least $\left(a - \frac{1}{4}\right)^{1/2}$.
\end{lemma}

If we exclude characteristics of the form $4k+3$, this Solymosi's result can be sharpened slightly.

\begin{lemma}
	\label{lem:triangle_4k_plus_3}
	Any triangle with sides $a \leq b \leq c$, where $c=a+b-1$ and $a, b, c \in \mathbb{N}$,
	has the characteristic of the form $p=4k+3$, $k\in \mathbb{N}_0$.
\end{lemma}

\begin{proof}
	Indeed, consider the triangle $ABC$ and let $|BC|=a$, $|AC|=b$.
	Using the triangle inequality $|AB| < |BC|+|AC|$,
	we can represent the length of side $AB$ as $|AB|=a+b-s$, $s \in \mathbb{N}$ and $s \leq a$.
	For $s=1$, the conditions of Lemma~\ref{lem:pre_weeding_I} are satisfied: $s$ is an odd number and $s \leq a$.
	Therefore, the triangle with sides $a$, $b$, and $c=a+b-1$
	has a characteristic of the form $p=4k+3$, $k\in \mathbb{N}_0$.
\end{proof}

\begin{lemma}
	\label{lem:triangle_height_our}
	Let a triangle $ABC$ have the characteristic different from $4k+3$, $k\in \mathbb{N}_0$, with $a\leq b\leq c$. Then the smallest height of the triangle $ABC$ is at least $(2a-1)^{1/2}$.
\end{lemma}

\begin{proof}
	By Lemma~\ref{lem:triangle_4k_plus_3}, a triangle with a side $c=a+b-1$ cannot have a characteristic different from $4k+3$, $k\in \mathbb{N}_0$. Therefore, in a triangle with integer sides, we have $a+b\geq c+2.$

	The height $h$ of the triangle $ABC$, dropped onto side $c$, can be found from the formula for its area: $S=hc/2$, which gives $h=2S/c$. To find the area of the triangle, we use the Heron's formula in the following form:
	$$
	S=\frac{1}{4} \sqrt{4a^2c^2-(c^2+a^2-b^2)^2}.
	$$
	Then
	\begin{multline}
		\label{eq:39}
		h^2=\left( \frac{2}{4c} \cdot \sqrt{4a^2c^2-(c^2+a^2-b^2)^2}\right)^2
		=
		\frac{1}{4c^2} \cdot \left( 4a^2c^2-(c^2+a^2-b^2)^2 \right)
		\\
		=
		a^2- \left(\frac{c^2+a^2-b^2}{2c} \right)^2
	\end{multline}

	From three heights of the triangle, the smallest one is that dropped onto its largest side.
	For fixed sides $a$ and $b$, let us set $c=a+b-2$. Then
	\begin{multline}
		\label{eq: 40}
		\frac{c^2+a^2-b^2}{c}
		=
		c+\frac{a+b}{c}(a-b) = c+\frac{c+2}{c}(a-b)
		=
		\\
		c+\left(1+\frac{2}{c}\right)(a-b) \leq c+a-b=2a-1
	\end{multline}
	and this equality is possible when $a=b$. We rewrite expression~\eqref{eq:39} as follows:
	$$
	h^2 \geq a^2- \left(\frac{2a-1}{2} \right)^2=a^2-(a-1)^2=a^2-a^2+2a-1=2a-1.
	$$

	Thus, the smallest height of the triangle with characteristic different from $4k+3$ is at least $(2a-1)^{1/2}$.
\end{proof}

\begin{corollary}
	\label{cor:sqrt_5_3}
	For the height of a triangle $ABC$, where ${A, B, C} \subset M \in \overline{\mathfrak{M}_4}'$, with sides $3 \leq a \leq b \leq c$, the following estimate holds:
	$$
		(2a-1)^{1/2} \geq \frac{\sqrt{5}}{\sqrt{3}} a^{1/2},
	$$
	where the difference between the left and right-hand sides increases with increasing $a$.
\end{corollary}

\section{The Main Bound}

The following lemma is a classical Erdos-style intersection-enumeration one,
empowered by our Weeding Theorem I.

\begin{lemma}
	\label{lem:intersection_enumeration}
	Let ${M_1, M_2, M_3, M_4} \subset M \in \overline{\mathfrak{M}_n}$ (points $M_2$ and $M_3$ may coincide, while the others are distinct), where $n \ge 4$. Then $\# M \le |M_1 M_2| \cdot |M_3 M_4| - 2$.
\end{lemma}

\begin{proof}
	We will distinguish three cases: two even edges, two odd edges and two edges with different parity.

	Let's consider the first case. Suppose both edges $M_1M_2$ and $M_3M_4$ have even lengths. Then, for each point $N\in M$, one of the following conditions is satisfied:

	a) $N$ belongs to $cr(M_1,M_2)$, which implies that there are no more than 4 points (no more than 2 on each of the lines);

	b) $N$ belongs to $cr(M_3,M_4)$, which implies that there are no more than 4 points (no more than 2 on each of the lines);

	c) $N$ belongs to the intersection of one of the $(|M_1M_2|/2 - 1)$ hyperbolae with one of the $(|M_3M_4|/2 - 1)$ hyperbolae, which implies that there are no more than $4(|M_1M_2|/2 - 1)(|M_3M_4|/2 - 1)$ points.

	Assuming that the edges $M_1M_2$ and $M_3M_4$ have even lengths,
	we infer that $|M_1M_2|\geq4$ and $|M_3M_4|\geq4$.
	Then,
	\begin{multline}
		4 \left(\frac{|M_1 M_2|}{2} - 1\right) \left(\frac{|M_3 M_4|}{2} - 1 \right) + 4 + 4
		=
		\\=
		 4\left(\frac{|M_1 M_2| \cdot |M_3 M_4|}{4} - \frac{|M_1 M_2|}{2} - \frac{|M_3 M_4|}{2}+1\right) + 4 + 4
		=
		\\=
		|M_1 M_2| \cdot |M_3 M_4| - 2 |M_1 M_2| - 2 |M_3 M_4| + 4 + 4 + 4
		\leq
		|M_1 M_2| \cdot |M_3 M_4| - 4 \\
		<
		|M_1 M_2| \cdot |M_3 M_4| - 2
		.
	\end{multline}

	Let's consider the second case. Suppose both edges $M_1M_2$ and $M_3M_4$ have odd lengths. Then, for each point $N\in M$, one of the following conditions is satisfied:

	a) $N$ belongs to $cr(M_1,M_2)$, which implies that there are no more than 2 points (otherwise, $\operatorname{char} M = 4k+3$);

	b) $N$ belongs to $cr(M_3,M_4)$, which implies that there are no more than 2 points;

	c) $N$ belongs to the intersection of one of $(|M_1 M_2|-1)/2$ hyperbolae with one of $(|M_3M_4|-1)/2$ hyperbolae, which implies that there are no more than $(|M_1 M_2|-1)(|M_3 M_4|-1)$ points.

	Assuming that the edges $M_1M_2$ and $M_3M_4$ have odd lengths,
	we infer that $|M_1M_2|\geq3$ and $|M_3M_4|\geq3$.
	Then,
	\begin{multline}
		(|M_1 M_2| - 1)(|M_3 M_4| - 1) + 2 + 2
		=
		|M_1 M_2| \cdot |M_3 M_4| - |M_1 M_2| - |M_3 M_4| + 4
		\leq \\
		\leq |M_1 M_2| \cdot |M_3 M_4| - 2
		.
	\end{multline}

	Let's consider the third case. Suppose the lengths of edges $M_1M_2$ and $M_3M_4$ are different. Without loss of generality suppose the edge $M_1M_2$ has even length while edge $M_3 M_4$ has odd length.
	Then, for each point $N\in M$, one of the following conditions is satisfied:

	a) $N$ belongs to $cr(M_1, M_2)$, which implies that there are no more than 4 points (no more than 2 on each of the lines);

	b) $N$ belongs to $cr(M_3,M_4)$, which implies that there are no more than 2 points (otherwise $\operatorname{char} M = 4k+3$);

	c) $N$ belongs to the intersection of one of $(|{M_1 M_2}|/2-1)$ hyperbolae with one of $(|M_3 M_4|-1)/2$ hyperbolae, which implies that there are no more than $4(|M_1 M_2|/2-1)(|M_3 M_4|-1)/2$ points.

	From the parity assumption for edges $M_1 M_2$ and $M_3 M_4$
	we can infer that $|M_1 M_2|\geq4$ and $|M_3M_4|\geq3$.
	Then,
	\begin{multline}
		4\left(\frac{|M_1 M_2|}{2} - 1\right) \left(\frac{|M_3 M_4| - 1}{2}\right) + 2 + 4
		=
		\\=
		(|M_1 M_2| - 2)(|M_3 M_4| - 1) + 6
		=
		\\=
		|M_1 M_2| \cdot |M_3 M_4| - |M_1 M_2| - 2 |M_3 M_4| + 2 + 6
		\leq
		|M_1 M_2| \cdot |M_3 M_4| - 2
		.
	\end{multline}

	This proves the assertion.
\end{proof}

The following function was introduced in~\cite{kurz2008minimum}
\begin{equation}
	\overline{d}(2,n) = \min_{M\in\overline{\mathfrak{M}_{n}}} \operatorname{diam} M
	,
\end{equation}
and some values were given:

\begin{eqnarray}
\label{d}
\overline{d}(2,3)&=& 1,{}\nonumber\\
\overline{d}(2,4)&=& 4,{}\nonumber\\
\overline{d}(2,5)&=&  \overline{d}(2,6)= 8,{}\nonumber\\
\overline{d}(2,7)&=& 33,{}\nonumber\\
\overline{d}(2,8)&=&  \overline{d}(2,9)= 56,{}\nonumber\\
\overline{d}(2,10)&=& ... =\overline{d}(2,12)= 105,{}\nonumber\\
\overline{d}(2,13)&=&  \overline{d}(2,14)= 532,{}\nonumber\\
\overline{d}(2,15)&=&  ... = \overline{d}(2,18) = 735,{}\nonumber\\
\overline{d}(2,19)&=& ... = \overline{d}(2,24) = 1995,{}\nonumber\\
\overline{d}(2,25)&=&  ... = \overline{d}(2,27)= 9555,{}\nonumber\\
\overline{d}(2,28)&=& 10672,{}\nonumber\\
\overline{d}(2,29)&=&  ... = \overline{d}(2,36) = 13975,{}\nonumber\\
\overline{d}(2,37) &>& 20\ 000.
\end{eqnarray}

Now we are finally ready to proof our main result.
It improves the bound~\eqref{eq:chebsb_lower_bound}
in our special case.

\begin{theorem}
	\label{thm:main_result}
	Let $M \in \overline{\mathfrak{M}}_n'$, i.e., $M$ is a set in semi-general position with characteristic different from $4k+3$, $k \in \mathbb{N}_0$.
	Then, for every integer $n \geq 3$, the inequality
	\begin{equation}
		\operatorname{diam} M \geq \left( \frac{25n}{36} \right)^{5/4}
	\end{equation}
	holds.
\end{theorem}

\begin{proof}
	For $n = 3$, we have $\operatorname{diam} M \geq 3$ (achieved by an isosceles triangle with sides 2, 3, 3), and the assertion is obvious.

	Let us consider $M \in \overline{\mathfrak{M}_n}'$, $n \geq 4$, $\operatorname{diam} M = p$.

	Choose points $M_1, M_2, M_3, M_4 \in M$ (points $M_2$ and $M_3$ may coincide, while the others must be pairwise distinct) such that

	\begin{equation*}
		\min_{A,B \in M} |AB| = |M_1 M_2|
	\end{equation*}

	\begin{equation*}
		\min_{A,B \in M \setminus {M_1}} |AB| = |M_3 M_4| = m
	\end{equation*}

	If $m \leq \frac{6}{5}p^{2/5}$, then by Lemma~\ref{lem:intersection_enumeration},
	\begin{equation*}
		n \leq |M_1 M_2| \cdot |M_3 M_4| - 2 \leq \frac{36}{25}p^{4/5} - 2
	\end{equation*}
	or equivalently
	\begin{equation}
		p \geq \left(\frac{25(n+2)}{36}\right)^{5/4} \geq \left(\frac{25n}{36}\right)^{5/4}
	\end{equation}
	which is exactly the claim of the theorem.

	Now we have to consider $m > \frac{6}{5}p^{2/5}$.
	Then for any points $A, B \in M \setminus {M_1}$, we have $|AB| > \frac{6}{5}p^{2/5}$.
	By Corollary~\ref{cor:sqrt_5_3} and Lemma~\ref{lem:smurov_min_height},
	no three points from $M \setminus {M_1}$ lie in a strip of width
	\begin{equation}
		\frac{\sqrt{5}}{\sqrt{3}} \cdot \sqrt{\frac{6p^{2/5}}{5}} = \sqrt{2} \cdot p^{1/5}
		.
	\end{equation}

	By Lemma~\ref{lem:square_container}, the set $M$ lies in a square with side length $p$. We cover this square by $q$ strips, $\frac{p^{4/5}}{\sqrt{2}} \leq q < \frac{p^{4/5}}{\sqrt{2}} + 1$, such that the width of each strip does not exceed $\sqrt{2} \cdot p^{1/5}$.
	Each of the obtained strips contains no more than two points from $M \setminus {M_1}$, so
	\begin{equation}
		\label{eq:46}
		n \leq 2 \left( \frac{p^{4/5}}{\sqrt{2}} + 1\right) + 1=\frac{2p^{4/5}}{\sqrt{2}} +3 = \sqrt{2}p^{4/5}+3
		.
	\end{equation}

	From inequality~\eqref{eq:46}, we obtain
	\begin{equation}
		\label{eq:strips_5_4}
		p \geq \left( \frac{n -3}{\sqrt{2}} \right)^{5/4}
		.
	\end{equation}

	According to the results~\eqref{d}, for $3 \leq n \leq 36$ the theorem is true.
	Moreover, it is known that for all $37 \leq n \leq 7^4$, we have $\overline{d}(2, n) > 20\ 000$  and our estimate on the diameter is also true.
	Indeed, $\overline{d}(2, 7^4) \geq \left(\frac{25}{36}\cdot 7^4\right)^{5/4} \approx 10,655$.
	This estimate is weaker than the available numerical results. Therefore, from now on, we can assume that $n \geq 7^4$.

	For estimation~\eqref{eq:strips_5_4} and $n \geq 7^4$, we have
	\begin{equation}
		\label{eq:strips_5_4_7}
		p \geq \left( \frac{n -3}{\sqrt{2}} \right)^{5/4}  \geq \left( \frac{25n}{36} \right)^{5/4}
		.
	\end{equation}

	Thus, for any $n \geq 3$ the inequality $
		\operatorname{diam} M \geq \left( \frac{25n}{36} \right)^{5/4}
	$ holds.
\end{proof}

\section{Conclusion}
	The bound proved above (as well as~\eqref{eq:chebsb_lower_bound}) may appear to be far from precise values of $\overline{d}(2,n)$.
	However, it's easy to see that values of $d(2;n)$ tend to repeat often;
	thus, it is rather not unrealistic that the bounds may converge to the precise values.

	Also, we should notice that our approach and  specifically Weeding Theorems do not require semi-general position and can be applied to tighten the bound from~\cite{my-pps-linear-bound-2019}
	in the special case of characteristic $4k+1$ or $4k+2$.
	However, we are not ready to accept this tedious challenge yet.

	Another research area that is able to utilize our results fruitfully is the maximization of IPS.
	The Weeding Theorems I and II can sometimes make the exhaustive search up to 4 times faster.

\section{Acknowledgements}
The author thanks Dr. Prof. E.M. Semenov for the fruitful discussion and ideas,
Dr. A.S. Chervinskaia for the idea of using the term ``facher'' and proofreading,
Dr.~A.S.~Usachev for proofreading,
and E.A. Momot for the ``SciLexic'' project that helped the author with some word usage issues.




\begin{thebibliography}{99}
%
\bibitem{anning1945integral}
N. H. Anning and P. Erdős, Integral distances, \emph{Bull. Amer. Math. Soc.} 51.\textbf{8} (1945), 598–600, \textsc{doi}: \href
{https://doi.org/10.1090/S0002-9904-1945-08407-9} {\nolinkurl {10.1090/S0002-9904-1945-08407-9}}.
%
\bibitem{antonov2008maximal}
A. R. Antonov and S. Kurz, Maximal integral point sets over {$\mathbb {Z}^2$}, \emph{Int. J. Comput. Math.} 87.\textbf{12}
(2008), 2653–2676, \textsc{doi}: \href {https://doi.org/10.1080/00207160902993636} {\nolinkurl {10.1080/00207160902993636}},
arXiv: \href {http://arxiv.org/abs/0804.1280} {\nolinkurl {0804.1280}}.
%
\bibitem{our-vmmsh-2018-translit}
N. N. Avdeev and E. M. Semenov, Integral point sets in the plane and Euclidean space (Множества точек с целочисленными
расстояниями на плоскости и в евклидовом пространстве), \emph{Matematicheskij forum (Itogi nauki. Yug Rossii)} (2018),
217–236.
%
\bibitem{avdeev2019particular}
N. N. Avdeev, R. E. Zvolinsky, and E. A. Momot, On particular diameter bounds for integral point sets in higher dimensions
(2019), arXiv: \href {http://arxiv.org/abs/1909.10386} {\nolinkurl {1909.10386}}, \textsc{url}: \url
{https://ui.adsabs.harvard.edu/abs/2019arXiv190910386A/abstract}.
%
\bibitem{my-semi-general-5-4-bound-2019}
N. Avdeev, On diameter bounds for planar integral point sets in semi-general position (2019), arXiv: \href
{http://arxiv.org/abs/1907.09331} {\nolinkurl {1907.09331}}, \textsc{url}: \url
{https://ui.adsabs.harvard.edu/abs/2019arXiv190709331A/abstract}.%
\bibitem{my-pps-linear-bound-2019}
N. Avdeev, On existence of integral point sets and their diameter bounds, \emph{Australas. J. Combin.} 77.\textbf{1} (2020),
100–116, arXiv: \href {http://arxiv.org/abs/1906.11926} {\nolinkurl {1906.11926}}, \textsc{url}: \url
{https://ui.adsabs.harvard.edu/abs/2019arXiv190611926A/abstract}.
%
\bibitem{costa2013valid}
A. Costa, Valid constraints for the point packing in a square problem, \emph{Discrete Appl. Math.} 161.\textbf{18} (2013),
2901–2909.
%
\bibitem{erdos1945integral}
P. Erdős, Integral distances, \emph{Bull. Amer. Math. Soc.} 51.\textbf{12} (1945), 996, \textsc{doi}: \href
{https://doi.org/10.1090/S0002-9904-1945-08490-0} {\nolinkurl {10.1090/S0002-9904-1945-08490-0}}.
%
\bibitem{harborth1993upper}
H. Harborth, A. Kemnitz, and M. Möller, An upper bound for the minimum diameter of integral point sets, \emph{Discrete
Comput. Geom.} 9.\textbf{4} (1993), 427–432, \textsc{doi}: \href {https://doi.org/10.1007/bf02189331} {\nolinkurl
{10.1007/bf02189331}}.
%
\bibitem{huff1948diophantine}
G. B. Huff, Diophantine problems in geometry and elliptic ternary forms, \emph{Duke Math. J.} 15.\textbf{2} (1948), 443–453,
\textsc{doi}: \href {https://doi.org/10.1215/S0012-7094-48-01543-9} {\nolinkurl {10.1215/S0012-7094-48-01543-9}}.
%
\bibitem{kemnitz1988punktmengen}
A. Kemnitz, Punktmengen mit ganzzahligen Abständen, 1988.
%
\bibitem{kreisel2008heptagon}
T. Kreisel and S. Kurz, There are integral heptagons, no three points on a line, no four on a circle, \emph{Discrete Comput.
Geom.} 39.\textbf{4} (2008), 786–790, \textsc{doi}: \href {https://doi.org/10.1007/s00454-007-9038-6} {\nolinkurl
{10.1007/s00454-007-9038-6}}.
%
\bibitem{kurz2005characteristic}
S. Kurz, On the characteristic of integral point sets in {$\mathbb {E}^m$}, \emph{Australas. J. Combin.} 36 (2006), 241–248,
arXiv: \href {http://arxiv.org/abs/math/0511704} {\nolinkurl {math/0511704}}.
%
\bibitem{kurz2008bounds}
S. Kurz and R. Laue, Bounds for the minimum diameter of integral point sets, \emph{Australas. J. Combin.} 39 (2007), 233–240,
arXiv: \href {http://arxiv.org/abs/0804.1296} {\nolinkurl {0804.1296}}.
%
\bibitem{kurz2008minimum}
S. Kurz and A. Wassermann, On the minimum diameter of plane integral point sets, \emph{Ars Combin.} 101 (2011), 265–287,
arXiv: \href {http://arxiv.org/abs/0804.1307} {\nolinkurl {0804.1307}}.
%
\bibitem{kurz2013constructing}
S. Kurz et al., Constructing $7$-clusters, \emph{Serdica J. Comput.} 8.\textbf{1} (2014), 47–70, arXiv: \href
{http://arxiv.org/abs/1312.2318} {\nolinkurl {1312.2318}}.
%
\bibitem{markot2005newverified}
M. C. Markót and T. Csendes, A new verified optimization technique for the "packing circles in a unit square" problems,
\emph{SIAM J. Optim.} 16.\textbf{1} (2005), 193–219.
%
\bibitem{momot2022example}
E. A. Momot, R. E. Zvolinskiy, and A. E. Zvolinskiy, An example of rails integral point set with cardinality 106, Прикладная
математика и фундаментальная информатика, 2022, 114–115.
%
\bibitem{nozaki2013lower}
H. Nozaki, Lower bounds for the minimum diameter of integral point sets, \emph{Australas. J. Combin.} 56 (2013), 139–143.
%
\bibitem{solymosi2003note}
J. Solymosi, Note on integral distances, \emph{Discrete Comput. Geom.} 30.\textbf{2} (2003), 337–342, \textsc{doi}: \href
{https://doi.org/10.1007/s00454-003-0014-7} {\nolinkurl {10.1007/s00454-003-0014-7}}.
%
\bibitem{solymosi2010question}
J. Solymosi and F. De Zeeuw, On a question of Erdős and Ulam, \emph{Discrete Comput. Geom.} 43.\textbf{2} (2010), 393–401,
arXiv: \href {http://arxiv.org/abs/0806.3095} {\nolinkurl {0806.3095}}.
%
\bibitem{our-ped-2018}
Н. Н. Авдеев, On integral point sets in special position, \emph{Некоторые вопросы анализа, алгебры, геометрии и
математического образования: материалы международной молодежной научной школы «Актуальные направления
математического анализа и смежные вопросы»} 8 (2018), 5–6.%
\bibitem{our-pmm-2018}
Н. Н. Авдеев, Об отыскании целоудалённых множеств специального вида, Актуальные проблемы прикладной математики,
информатики и механики - сборник трудов Международной научной конференции. Научно-исследовательские публикации,
2018, 492–498.
%
\bibitem{our-vmmsh-2018}
Н. Н. Авдеев and Е. М. Семёнов, Множества точек с целочисленными расстояниями на плоскости и в евклидовом
пространстве, Russian, \emph{Математический форум (Итоги науки. Юг России)} (2018), 217–236.
%
\bibitem{smurov1998stripcoverings}
М. Смуров and А. Спивак, Покрытия полосками, \emph{Квант} \textbf{5} (1998), 6–12.
\end{thebibliography}
\end{document}